\documentclass[a4paper,12pt]{article}

\usepackage{amsmath,amsthm,amssymb,sansmath,color,geometry,multicol,graphicx, float,color,authblk, hyperref,bbm,stmaryrd,amsfonts}

\usepackage[english]{babel}
\usepackage[utf8]{inputenc}
\newtheorem{thm}{Theorem}

\newtheorem{lem}[thm]{Lemma}

\newcommand{\cco}{\llbracket}
\newcommand{\ccf}{\rrbracket}
\newcommand{\po}{\left(}
\newcommand{\pf}{\right)}

\newcommand{\R}{\mathbb R}

\newcommand{\E}{\mathbb E}
\newcommand{\na}{\nabla}
\newcommand{\dd}{\text{d}}

\newcommand{\ent}{\text{Ent}}
\newcommand{\Sp}{\mathcal S_N^{>0}\po\R\pf}
\newcommand{\An}{\mathcal A_N\po\R\pf}
\newcommand{\Spe}{\mathcal S_N^{\geqslant 0}\po\R\pf}

\title{Optimal linear drift for the speed of convergence of an hypoelliptic diffusion}
\author{Arnaud Guillin, Pierre Monmarché}

\begin{document}

\maketitle
\abstract{Among all generalized Ornstein-Uhlenbeck processes which sample the same invariant measure and for which the same amount of randomness (a $N$-dimensional Brownian motion) is injected in the system, we prove that the asymptotic rate of convergence is maximized by a non-reversible hypoelliptic one\footnote{This paper has been published in \emph{Electron. Commun. Probab. Vol. 21 (2016)}, followed by an erratum. Later on, some remaining typos in the last paragraph of Section 3 (description of the algorithm to construct an optimal drift)  were kindly pointed out to us by Lancelot Da Costa. The present version already contains the changes indicated in the erratum, and  the end of Section 3 has been corrected.}.}

\section{Introduction}

For a potential $V:\R^N\rightarrow \R$ such that $\int e^{-V}<\infty$, consider $\mu$ the associated Gibbs law, namely the  probability measure with a density with respect to the Lebesgue measure proportional to $e^{-V}$. In order to compute expectations with respect to $\mu$, Markov Chain Monte Carlo (MCMC) algorithms are widely spread. Such an algorithm is based on an ergodic Markov process $(X_t)_{t\geq0}$ whose unique invariant law is $\mu$, so that for $T$ large enough $X_T$ is not far to be distributed according to $\mu$. The efficiency of the algorithm is directly linked to the rate of convergence of $X$ toward its equilibrium, which is why a fair amount of work has been devoted to accelerating this convergence (see \cite{LelievreFreeEnergy} and references within). In particular, since there are many possible Markov processes to sample the same equilibrium $\mu$, the question arises to choose the fastest, if any.

Along with those obtained from a Metropolis-Hasting procedure (see e.g. \cite{Lelievre2006} and references within), one of the most classical Gibbs sampler is the Fokker-Planck diffusion that solves the SDE
\begin{eqnarray}\label{EquaFP}
\dd X_t &=& -\na V(X_t)\dd t + \sqrt{2}\dd B_t
\end{eqnarray}
where $B$ is a standard $N$-dimensional Brownian motion. Its generator is
\begin{eqnarray*}
L&=& -\na V\cdot \na  + \Delta 
\end{eqnarray*}
where we recall the generator of a Markov process $X$ is formally defined by
\begin{eqnarray*}
Lf(x) &=& \po\partial_t\pf_{|t=0} \E\po f(X_t)\ |\ X_0=x\pf.
\end{eqnarray*}
The Fokker-Planck diffusion is a reversible process in the sense its generator is self-adjoint in $L^2(\mu)$. This property is of theoretical interest but, from a practical point of view, reversible processes are usually not optimal with regard to their speed of convergence. A particular technique for improving the convergence of $X$ to $\mu$ is to add a divergence-free (with respect to $\mu$) drift $b$, namely to consider the SDE
\begin{eqnarray}\label{EquaIrreversible}
\dd X_t &=& \po-\na V(X_t)+b(X_t)\pf\dd t + \sqrt{2}\dd B_t
\end{eqnarray}
with $b$ such that $\na\cdot\po b e^{-V}\pf = 0$ where $\na\cdot$ stands for the divergence operator. That way, the equilibrium is not affected, but the process is no longer reversible and the convergence is improved (cf. \cite{HHS1,HHS2,Arnold1,Lelievre2012}). It can be easily seen if we consider for example convergence in $L^2(\mu)$, where the spectral gap of the reversible dynamic is a lower bound for the speed of convergence of the non reversible one (just by comparison of Dirichlet forms), see \cite{HHS2}.

Another way to improve the convergence is to consider a kinetic process $(X,Y)$ where $X$ is the position and $Y=\dd X/\dd t$ is the velocity, which acts as an instantaneous memory (see \cite{Gadat2013,Lelievre2006,DiaconisMiclo}). For instance the Langevin diffusion
\begin{eqnarray}\label{EquaLangevin}
 & &  \left\{\begin{array}{rcl}
\dd X_t &=& Y_t \dd t\\
\dd Y_t &=& -\na V(X_t) \dd t - Y_t \dd t+ \sqrt 2 \dd B_t 
\end{array}\right.
\end{eqnarray}
admits $e^{-H}$ as an invariant measure where the Hamiltonian is $H(x,y) = V(x) + \frac12 |y|^2$. In particular, the first marginal of this equilibrium is $\mu$. The Langevin diffusion is non-reversible and moreover it is hypoelliptic. It has been observed in \cite{Lelievre2006} that it may converge faster than the reversible Fokker-Planck diffusion in some applied problems. It recently regained much interest under the name Hamiltonian  Monte Carlo methods \cite{GC11}.

For both dynamics \eqref{EquaIrreversible} and \eqref{EquaLangevin} it is difficult for a general potential $V$ to obtain sharp theoretical bounds on the rates of convergence (see \cite{MonmarcheRecuitHypo} for consideration on this matter in the metastable case, namely the regime $\varepsilon\rightarrow 0$ with the potential $V_\varepsilon = \frac1\varepsilon V$ where $V$ has several local minima). A particular simple situations is the case where $V$ is quadratic or, in other words,  $\mu$ is a Gaussian measure. Of course MCMC algorithms are not really relevant in practice regarding sampling according to Gaussian measures, but then the exact rates of convergence for \eqref{EquaIrreversible} and \eqref{EquaLangevin} are trackable (see \cite{Arnold2014,MonmarcheGamma} and  below).

In this context, the purpose of the present work is to answer the problem raised in \cite{Lelievre2012}, namely: for a given Gaussian law $\mu$, is it possible to find the optimal divergence-free linear drift one can add to \eqref{EquaLangevin} in order to obtain the largest rate of convergence ? More generally, what is the largest rate of convergence one can get when sampling according to $\mu$  using a (possibly hypoelliptic) Markov diffusion with linear drift and constant diffusion coefficients ?

Obviously this question is ill-posed since the invariant measure of $(W_t)_{t\geq 0}=\po X_{\lambda t}\pf_{t\geq0}$ is still $\mu$ for any $\lambda>0$, and $W$ goes $\lambda$ times faster than $X$ to equilibrium. Following \cite{Gadat2013} we will thus work under the additional assumption that the total amount of randomness instantaneously injected in the system (that is, the trace of the diffusion matrix) is prescribed.

In the following, we will first introduce the main notations and recall basic facts about generalized Ornstein-Uhlenbeck processes. In Section 2, we present our main results, giving a positive and definite answer to the problem raised in \cite{Lelievre2012}. Section 3 is dedicated to the proofs of our main results, whereas Section 4 presents numerical illustration of our results and present some thoughts on the general case we wish to tackle in the future.

\subsubsection*{Notations}

In this whole work,  $\mathcal M_{N}\po\R\pf$ is the set of $N\times N$ real matrices, $\Sp$ (resp. $\Spe$) the set of positive definite (resp. semi-definite)  symmetric ones and $\An$ is the set of anti-symmetric ones. The spectrum of a matrix $A$ is $\sigma(A)$, its trace is $Tr(A)$, its transpose is $A^T$ and vectors are considered as column matrices, so that the scalar product $x\cdot y$ is $x^T y$. 
 Finally $\Re(\lambda)$ stands for the real part of $\lambda\in\mathbb C$ and $diag(a_1,\dots,a_N)$ stands for the the diagonal matrix with coefficients $a_i$.

\subsubsection*{Basic facts about Ornstein-Uhlenbeck processes}

We recall here some facts whose proofs and details may be found for instance in  \cite{Arnold2014}. A generalized Ornstein-Uhlenbeck  process (OUP) is any diffusion with a linear drift and a constant matrix diffusion. In other words in dimension $N$ it is the solution of an SDE of the form
\begin{eqnarray*}
\dd X_t &=& A X_t \dd t + \sum_{j = 0}^N \sigma_j \dd B^j_t
\end{eqnarray*}
where $A$ is a constant matrix, the $\sigma_j$'s are constant vectors and the $B^j$'s are 1-dimensional independent Brownian motions. A Markov process is an OUP if and only if its generator is of the form
\begin{eqnarray}\label{DefGeneLAD}
L_{A,D} f(x) &:=& (Ax)^T \na f(x) + \na\cdot \po D \na f\pf(x)
\end{eqnarray}
where $D = \frac12\sum \sigma_j \sigma_j^T$ is a positive semi-definite matrix and $\na\cdot$  stands for the divergence operator. Recall that a measure $\mu$ is said invariant for $X$ (or equivalently for $L_{A,D}$) if $Law\po X_0\pf = \mu$ implies $Law\po X_t\pf = \mu$ for all $t\geq 0$. For an OUP, an invariant measure 
 is necessarily a (possibly degenerated) Gaussian distribution.

 On the other hand the process is hypoelliptic if and only if Ker$D$ does not contain any non-trivial subspace which is invariant by $A^T$, and in that case an invariant measure is necessarily unique and non-degenerated (it has a positive density with respect to the Lebesgue measure on $\R^N$).  In the case where the invariant measure exists and is unique, its density $\psi_\infty$ is the unique solution of $L_{A,D}'\psi = 0$ where 
\begin{eqnarray*}
L_{A,D}' f(x) &=& -(Ax)^T \na f(x) - \text{Tr}( A) f(x) + \na\cdot \po D \na f\pf(x)
\end{eqnarray*}
is the dual in the Lebesgue sense of $L_{A,D}$. 

We will focus mainly on generalized Ornstein-Uhlenbeck processes which have a non degenerate Gaussian distribution as invariant probability measure. Let $N\geq1$, $S\in\Sp$ and
\begin{eqnarray*}
\psi_\infty(x) &=& \frac{\po \det S\pf^{\frac12}}{\po 2\pi\pf^{\frac N2}} \exp\po \frac{- x^T S x}2\pf
\end{eqnarray*}
be the density of the (non-degenerated) Gaussian distribution with covariance matrix $S^{-1}$. The dual of $L_{A,D}$ in $L^2\po \psi_\infty\pf$ is then
\begin{eqnarray*}
L_{A,D}^* f(x) &=& \frac{1}{\psi_\infty(x)}L'_{A,D}\po f \psi_\infty\pf(x)\\
&=& -\po(2DS+A)x\pf^T \na f(x) + \na\cdot \po D \na f\pf(x)\\
&=& L_{-(2DS+A),D} f(x).
\end{eqnarray*}
Starting from an initial distribution $\psi_0$, the law $\psi_t$ and the density with respect to equilibrium $\frac{\psi_t}{\psi_\infty}$ at time $t$ of an OUP generated by $L_{A,D}$ are (weak) solutions of
\[\partial_t \psi_t = L_{A,D}'\psi_t\hspace{20pt}\text{and}\hspace{20pt} \partial_t \po \frac{\psi_t}{\psi_\infty} \pf = L_{A,D}^* \po \frac{\psi_t}{\psi_\infty} \pf.\]

\section{Main results}\label{SectionMainResult}

Let
\[\mathcal I(S) = \left\{ (A,D)\in \mathcal M_{N\times N}\po\R\pf \times \Spe,\ \text{Tr}D\leq N,\ L_{A,D}' \psi_\infty = 0\right\}\]
be the set of drift/diffusions matrices such that $\psi_\infty$ is invariant for the corresponding OUP and with at most the same amount of randomness injected in the system as the reversible dynamics with generator
\begin{eqnarray*}
L_{-S,I_N} f(x) &=& - \po S x\pf^T \na f(x) + \Delta f(x).
\end{eqnarray*}
 For $A\in\mathcal M_{d\times d}\po\R\pf $ we write
\begin{eqnarray*}
\rho(A) &=& \inf\{-\Re(\lambda),\ \lambda\in\sigma(A)\}.
\end{eqnarray*}
As we will see later, if $(A,D)\in\mathcal I(S)$ with $\rho(A)>0$ then $\psi_t \rightarrow \psi_\infty$ as $t\rightarrow \infty$ for all $\psi_0$, which implies $\psi_\infty$ to be the unique invariant measure of $L_{A,D}$, which is by consequence necessarily hypoelliptic.

Our main result identifies the maximum of $\rho(A)$ under the constraint $(A,D)\in {\mathcal I}(S)$. 
\begin{thm}\label{TheoPrincipal}
For all $S\in\Sp$,
\begin{eqnarray*}
\max \left\{ \rho(A),\  (A,D)\in \mathcal I(S)\right\}  &=& \max \sigma(S).
\end{eqnarray*}
\end{thm}

\bigskip

For an OUP with drift matrix $A$ the rate of convergence to equilibrium is $\rho(A)$ (see \cite{MonmarcheGamma} and below). Hence Theorem \ref{TheoPrincipal} states that it is possible to sample an OUP that converges at rate $\max \sigma(S)$ to $\psi_\infty$ using the same amount of randomness as the classical reversible dynamics with generator $L_{-S,I_N}$, while the latter converges at rate
\[\rho(-S) = \min \sigma(S).\]
This should be compared to the results of Hwang, Hwang-Ma and Sheu \cite{HHS1} or of Leli\`evre, Nier and Pavliotis \cite{Lelievre2012} (the first one being anterior, and the second one giving more explicit bounds for the convergence) that reads
\begin{eqnarray*}
\max \left\{ \rho(A),\  A\text{ s.t. }(A,I_N)\in \mathcal I(S)\right\}   &=& \frac{\text{Tr} S}{N}
\end{eqnarray*}
which is the arithmetic mean of all eigenvalues of $S$. On the other hand, for $(A,D)\in \mathcal I(S)$, the process is reversible if and only if $A=-(2DS+A)$, namely $A=-DS$, and for $D = \frac{N}{\text{Tr} S^{-1}} S^{-1}$ this gives  
\begin{eqnarray*}
\max \left\{ \rho(A),\  (A,D)\in \mathcal I(S),\ L_{A,D}^*=L_{A,D}\right\}   &=&  \frac{N}{\text{Tr}S^{-1}},
\end{eqnarray*}
which is the harmonic mean of the eigenvalues of $S$. Finally, considering $(1-\varepsilon) L_{A,D} + \varepsilon L_{-S,I_N}$ for any arbitrary $\varepsilon$,
\begin{eqnarray*}
\sup\left\{ \rho(A),\  (A,D)\in \mathcal I(S),\ D\text{ invertible}\right\}  &=& \max \sigma(S).
\end{eqnarray*}
 Note that
\[ \min \sigma(S)\hspace{10pt} \leq \hspace{10pt} \frac{N}{\text{Tr}S^{-1}} \hspace{10pt} \leq  \hspace{10pt} \frac{\text{Tr} S}{N} \hspace{10pt} \leq \hspace{10pt} \max \sigma(S) \]
and that the equalities hold only when $S$ is a homogeneous dilation, in which case
  no non-reversible dynamics can yield any improvement of the rate of convergence to equilibrium. On the other hand when the eigenvalues have different orders of magnitude (which means the problem is multi-scale; cf. \cite[Fig. 4]{Lelievre2012} where $S$ has uniformly distributed coefficients in $[0,1]$), the improvement is already clear from the reversible diffusion (with $D=I_N$) to the non-reversible (but still elliptic) ones, but yet it seems even more drastic when hypoelliptic dynamics are allowed. Obviously, the cost to pay for an optimal asymptotic speed of convergence is an initial delay for small times.Moreover, due to possibly large coefficients in the drift, a theoretically optimal continuous-time diffusion may lead, after discretization, to a not-so-efficient algorithm implemented in practice. We will leave aside this consideration, and only concentrate on the continuous-time problem.

We will focus here on the convergence in the entropy sense, ensuring for example also convergence in total variation via Pinsker's inequality. However, the same line of reasoning will also work for $L^2$ convergence or $\Phi-$entropies (see \cite{BolleyGentil}). More precisely, for a measure $\mu$, denote by
\begin{eqnarray*}
\ent_{\mu}(h) &:=& \int h \ln h \dd \mu - \po \int h \dd \mu \pf \ln \po \int h \dd \mu \pf 
\end{eqnarray*}
the entropy of a positive function $h$ with respect to $\mu$. For the reversible elliptic OUP with generator $L_{-S,I_N}$, it is well known that for all $h>0$
\begin{eqnarray*}
\ent_{\psi_\infty}\po e^{tL_{-S,I_N}^*} h \pf &\leq & e^{-2\rho(-S)t}\ent_{\psi_\infty}\po   h \pf.
\end{eqnarray*}
This is nothing else than an equivalent formulation of the Gaussian logarithmic Sobolev inequality of Nelson, see for example \cite{BGL} and references therein.

For a general OUP, if $(A,D)\in\mathcal I(S)$, according to \cite[Corollary 12]{MonmarcheGamma} there exists a constant $c\geq 1$ such that for all $h>0$
\begin{eqnarray*}
\ent_{\psi_\infty}\po e^{tL_{A,D}^*} h \pf &\leq & ce^{-2\rho(A)t}\ent_{\psi_\infty}\po   h \pf
\end{eqnarray*}
at least if $A$ is diagonalizable (when it is not the case, a polynomial in $t$ prefactor should be added). For a non-reversible yet elliptic OUP, $c$ may be strictly greater than 1 due to a change of norm, exactly as when we   consider  the transport semi-group $e^{tL_{A,0}} f = f\po e^{tA} \cdot \pf$ alone and write
\[ | e^{tA} x|^2 \leq |Q|^2|Q^{-1}e^{tA}Q|^2|Q^{-1} x|^2 \leq e^{-2\rho(A) t}|Q|^2|Q^{-1}|^2 |x|^2.    \]
 When the process is both non-reversible and non-elliptic, there are two reasons for $c$ to be greater than 1: the change of norm for $e^{tA}$, and the initial regularization which is really slower than in the elliptic case. Indeed, the part of $c$ which is due to slow regularization may badly behave with $N$. More precisely, since the optimal $(A,D)\in \mathcal I(S)$ we will consider will be very degenerated (the rank of $D$ being 1),  \cite[Remark p.16]{MonmarcheGamma} yields a constant $c$ of order $N^{40 N^2}$ which is, at the very least, absolutely awful. Of course this estimate is the result of a succession of rough bounds and a more careful (and involved) analysis could certainly refine it, but it is unclear whether the optimal bound is less than exponential with respect to $N$.

  Fortunately this problem disappears if we start the dynamics with an elliptic one and then switch to the hypoelliptic optimal one, ensuring thus first a quick regularization property.
 
 \begin{thm}\label{TheoPrincipal2}
 For any $C>1$ we can construct $(A,D)\in \mathcal I(S)$ such that for all $h>0$, with finite entropy, and for all $t,t_0>0$ with $t\geq t_0$,
 \begin{eqnarray*}
\text{\emph{Ent}}_{\psi_\infty}\po e^{(t-t_0)L_{A,D}^*} e^{t_0L_{-S,I_N}} h \pf &\leq & C \frac{ 1}{2 t_0 \min\sigma(S)}  e^{-2\po\max \sigma(S)\pf (t-t_0)}\text{\emph{Ent}}_{\psi_\infty}\po   h \pf.
\end{eqnarray*}
Moreover it is possible to construct $(A,D)\in \mathcal I(S)$ with $\| A\|_F \leq 4N^2\sqrt{ \frac{\po\max\sigma(S)\pf^3}{\min\sigma(S)}}$ (where $\| A\|_F = \sqrt{\text{Tr}\po A^T A\pf}$ is the Frobenius norm) 
such that for all $h>0$, with finite entropy, and for all $t\ge t_0>0$
 \begin{eqnarray}\label{EquationErratum}
\text{\emph{Ent}}_{\psi_\infty}\po e^{(t-t_0)L_{A,D}^*} e^{t_0L_{-S,I_N}} h \pf &\leq &  \frac{ 1 }{ t_0 \min\sigma(S) }  e^{-2\po\max \sigma(S)\pf (t-t_0)}\text{\emph{Ent}}_{\psi_\infty}\po   h \pf.
\end{eqnarray}
 \end{thm}
It is thus possible to completely quantify the initial loss, which, due to the role of the initial warm up via the reversible diffusion, boils down to the change of norm in the energy.

\newpage

\textbf{Remarks:}

\begin{itemize}
\item  As a function of $t_0$, the rhs of \eqref{EquationErratum} is minimal for $t_0^{-1} = 2\max \sigma(S)$, for which 
 \begin{eqnarray*} 
\text{\emph{Ent}}_{\psi_\infty}\po e^{(t-t_0)L_{A,D}^*} e^{t_0L_{-S,I_N}} h \pf &\leq &  2e \frac{ \max \sigma(S) }{   \min\sigma(S) }  e^{-2t \max \sigma(S) }\text{\emph{Ent}}_{\psi_\infty}\po   h \pf.
\end{eqnarray*}
\item The proof furnishes an explicit, algorithmic construction of an optimal $(A,D)$, with $D$ of rank 1. More generally, it is not hard to see from the proof that $D$ can be chosen with a rank at most the dimension of  the eigenspace associated to $\max \sigma(S)$.
\item In order to sample according to an $N$-dimensional Gaussian measure $\psi_\infty$, we can always add an $(N+1)^{th}$ auxiliary variable  and sample, with a given amount of randomness and at a rate arbitrarily large, according to an $N+1$-dimensional Gaussian measure whose first $N$-dimensional marginal is $\psi_\infty$. Of course, this should be counteracted in practice by numerical problems due to the discretization of the dynamics.
\item In the same line, we can also consider the question of the kinetic process \eqref{EquaLangevin}. More precisely, given $S\in \Sp$, then the equilibrium of the $2N$-dimenstional process $(X,Y)$ that solves
\begin{eqnarray}\label{EqKineticGaussien} 
 & &  \left\{\begin{array}{rcl}
\dd X_t &=& Y_t \dd t\\
\dd Y_t &=& - \nu S X_t \dd t - \frac1\nu Y_t \dd t+ \sqrt 2 \dd B_t 
\end{array}\right.
\end{eqnarray}
is $\psi_\infty(\dd x)\otimes \gamma_{\nu}\po \dd y\pf$, where $\gamma_{\nu}(\dd y) = \frac{1}{\sqrt{2\pi\nu^2}}e^{-\frac1{2\nu^2}|y|^2}\dd y$ is the Gaussian distribution with covariance matrix $\nu^2 I_N$. If the aim is to sample according to $\psi_\infty$, then the choice for the second $N$-dimensional marginal is open. Hence, we have to chose $\nu$ in order to maximize the rate of convergence to equilibrium of the whole process.

If $v$ is an eigenvector of $S$ associated to an eigenvalue $\lambda$, then $(v, -r v)$ is an eigenvector of
\[ A =\begin{pmatrix}
0 & I_N\\
-\nu S & -\frac1\nu I_N
\end{pmatrix}\]
associated to the eigenvalue $-r$ if and only if $r^2 - \frac1\nu r + \lambda \nu=0$, namely
\begin{eqnarray*}
r & = & \frac1{2\nu}\po 1 \pm \sqrt{1 - 4\lambda \nu^3}\pf.
\end{eqnarray*}
When $4\nu^3 > \max \sigma\po S^{-1}\pf$, the rate of convergence of $(X,Y)$ is thus $\po 2\nu\pf^{-1}$, which goes to zero as $\nu$ goes to $\infty$. On the other hand, when $\nu$ goes to 0, this rate is equivalent to $\nu^2 \min \sigma(S)$, which again goes to 0.  This proves that it is not possible to reach an arbitrarily large rate of convergence with the dynamics \eqref{EqKineticGaussien}, and that the best choice for $\nu$, the variance of the velocity, is neither to be found at infinity nor at zero (which answers a question raised in \cite{MonmarcheRecuitHypo}). Of course this would be a completely different story if we were to consider
\begin{eqnarray*}
 & &  \left\{\begin{array}{rcl}
\dd X_t &=& \frac1\nu Y_t \dd t\\
\dd Y_t &=& - S X_t \dd t - \frac1\nu Y_t \dd t+ \sqrt 2 \dd B_t ,
\end{array}\right. 
\end{eqnarray*}
which also allows to sample according to $\psi_\infty$, but is no longer a kinetic process in the sense $Y$ is no more the velocity $\frac{\dd X}{\dd t}$.

The optimal value of $\nu$ in \eqref{EqKineticGaussien} is explicit if $S=\lambda I_N$ is a homothety. Indeed, in that case, for $4\lambda \nu^3 >1$  the rate $r(\nu)$ is $(2\nu)^{-1}$ and thus is decreasing, while for $4\lambda \nu^3 <1$, 
\begin{eqnarray*}
r(\nu) & = & \frac1{2\nu}\po 1 - \sqrt{1 - 4\lambda \nu^3}\pf\\
\Rightarrow\hspace{10pt} r'(\nu) &=& \frac{r\po\nu\pf}{\nu}\po -1 + \frac{6\lambda \nu^3}{\sqrt{1 - 4\lambda \nu^3}\po 1 - \sqrt{1 - 4\lambda \nu^3}\pf}\pf\\
 &=& \frac{r\po\nu\pf}{\nu}\po \frac12 + \frac3{2\sqrt{1 - 4\lambda \nu^3}}\pf\\
\end{eqnarray*}
and thus is increasing. Hence, for a homothety $S=\lambda I_N$, the rate of convergence to equilibrium in \eqref{EqKineticGaussien} is maximal for $\nu = (4\lambda)^{-\frac13}$, and in that case this optimal rate is $\po \lambda/2\pf^{\frac13}$. In comparison, the rate of convergence of 
\begin{eqnarray*}
 \dd X_t  &=& - \lambda X_t \dd t + \sqrt 2 \dd B_t
\end{eqnarray*}
is $\lambda$, which is better than $\po \lambda/2\pf^{\frac13}$ if and only if $\lambda > \frac{1}{\sqrt 2}\simeq 0.7$.

\end{itemize} 

\section{Proofs}

Let us start with an easy lemma which will enable us to characterize $A$ in the couple $(A,D)\in \mathcal I(S)$ once $D$ is fixed.

\begin{lem}\label{LemJ}
For $S\in\Sp$, the following are equivalent:
\begin{itemize}
\item  $(A,D)\in \mathcal I(S)$ 
\item $D\in\Spe$ with Tr$D\leq N$ and there exists $J\in\An$ such that
\begin{eqnarray*}
A &=& -(D+J)S.
\end{eqnarray*}
\end{itemize}
\end{lem}
\begin{proof}
Fix $D\in\Spe$ such that Tr$D\leq N$. First we see that 
\[L_{-DS,D}'\po \psi_\infty\pf = \na\cdot\po (DSx)\psi_\infty -(DSx)\psi_\infty \pf = 0.\]
As a consequence
\[\po L_{A,D}'-L_{-DS,D}'\pf\psi_\infty  = -\na\cdot\po (A+DS)x\psi_\infty\pf\]
and \cite[Lemma 1]{Lelievre2012} concludes.
\end{proof}

We may now proceed to the proof of our main result.

\begin{proof}[Proof of Theorem \ref{TheoPrincipal}]
Let $(A,D)\in \mathcal I(S)$, so that by the previous lemma \ref{LemJ}, $D\in\Spe$ such that Tr$D\leq N$, $J\in\An$ and $A=-(D+J)S$. Following \cite{Lelievre2012} we write
\begin{eqnarray*}
(D+J)S &=& (S)^{-\frac12}\po \widetilde D + \widetilde J\pf  (S)^{\frac12}
\end{eqnarray*}
where $\widetilde D = (S)^{\frac12} D (S)^{\frac12}$ and similarly for $\widetilde J$. Note that $M \leftrightarrow (S)^{\frac12} M(S)^{\frac12}$ is a bijection which leaves $\Spe$ and $\An$ invariant, and that
\begin{eqnarray*}
\rho(A) &=& \rho\po - \widetilde D - \widetilde J\pf.
\end{eqnarray*}
From \cite[Propositions 1 and 4]{Lelievre2012},
\begin{eqnarray*}
\max\left\{ \rho\po - \widetilde D - \widetilde J\pf,\ \widetilde J\in\An\right\} &=& \frac{\text{Tr}\widetilde D}{N}.
\end{eqnarray*}
One may object that \cite[Propositions 1 and 4]{Lelievre2012} are written for an invertible matrix, which is not necessarily the case for $\widetilde D$. It can be seen that this restriction is not necessary in the proof; or to save the reader from a careful check of these proofs, we may note that $\widetilde D + \varepsilon I_N$ falls within the scope of \cite{Lelievre2012}, which yields the same result. Hence
\begin{eqnarray*}
\max\left\{ \rho\po A\pf,\ (A,D)\in \mathcal I(S)\right\} &=& \frac{1}{N} \max\left\{ \text{Tr}\po (S)^{\frac12} D (S)^{\frac12} \pf,\ D\in\Spe,\ \text{Tr}D\leq N\right\}. 
\end{eqnarray*}
Let $Q$ be an orthonormal matrix such that $S^{\frac12}  = Q^T\Sigma Q$ with $\Sigma = diag(\sqrt{\lambda_1},\dots,\sqrt{\lambda_N})$.
\begin{eqnarray*}
\text{Tr}\po (S)^{\frac12} D (S)^{\frac12} \pf &=& \text{Tr}\po \Sigma Q D Q^T \Sigma\pf \\
& \leq & \max \{\lambda_i,\ i=1..N\}  \text{Tr}\po   Q D Q^T  \pf \\
&= & \max \sigma(S)  \text{Tr}\po    D  \pf .
\end{eqnarray*}
We have proved 
\begin{eqnarray*}
\max\left\{ \rho\po A\pf,\ (A,D)\in \mathcal I(S)\right\} &\leq & \max \sigma(S) 
\end{eqnarray*}
and to prove equality holds we need to find $D\in\Spe$ with Tr$D\leq N$ such that Tr$\widetilde D= N \max \sigma(S)$. Let $i\in\cco 1,N\ccf$ be such that $\lambda_i = \max \sigma(S)$, let $E_{i,i}$ be the $N\times N$ matrix with all coefficients being zero except the coefficient $(i,i)$ being equal to 1, and set $D=N Q^T E_{i,i} Q$. In other words $D=Nv_iv_i^T$ where $v_i$ is a normalized eigenvector of $S$ associated to $\lambda_i$. In that case
\begin{eqnarray*}
\widetilde D &=& N Q^T \Sigma Q Q^T E_{i,i} Q  Q^T \Sigma Q\\
&=& \lambda_i D
\end{eqnarray*}
which concludes the proof.
\end{proof}

\begin{proof}[Proof of Theorem \ref{TheoPrincipal2}]
Without loss of generality (up to an orthonormal change of variables) we can assume the vector $v$ with all coordinates being equal to 1 is an eigenvector of $S$ associated to $\lambda=\max \sigma (S)$. Let $D=v v^T$, which is the matrix with all coefficients equal to 1. Let $0<\nu_1<\nu_2<\dots < \nu_N$  (to be specified later) and $Q = diag(\nu_1,\dots,\nu_N)$. Note that the eigenvectors of $Q$ are obviously the canonical basis vectors $(e_1,\dots,e_N)$, which satisfy $e_i^T D e_i = 1 = \text{Tr}D/N$ for all $i$. Let $\widetilde J$ be the antisymmetric matrix defined by
\[\po \widetilde J\pf_{k,l} =  \frac{\nu_k+\nu_l}{\nu_k-\nu_l}\]
if $k\neq l$ and 0 else. According to \cite[Lemma 2 and Equation (38)]{Lelievre2012}, 
\begin{equation}\label{inter}
 Q \widetilde J- \widetilde J Q  = -DQ-QD +2Q
\end{equation}
and $\sigma\po D + \widetilde J \pf \subset 1+i\mathbb R$. Let
\[A=-\lambda S^{-\frac12}\po D + \widetilde J \pf S^{\frac12} = -\po D+ S^{-\frac12}\widetilde J S^{-\frac12}\pf S,\]
so that $(A,D)\in\mathcal I(S)$ according to Lemma \ref{LemJ}. Recall that $L_{A,D}^* = L_{C,D}$ with
\begin{eqnarray*}
C&:=&-2DS - A \\
&=& -\lambda S^{-\frac12}\po D - \widetilde J \pf S^{\frac12}.
\end{eqnarray*} 
and note that $\sigma\po D - \widetilde J \pf = \sigma\po D + \widetilde J \pf^T  \subset 1+i\mathbb R$.

Let $\alpha$ be the function $r>0\mapsto \alpha(r) = r \ln r $ (or  $r\in\R\mapsto \alpha(r) = \frac12 r^2 $ if one desires to deal with $L^2$ decay rather than entropy), so that $\alpha''(r)$ is $r^{-1}$ (or 1). According to \cite[Lemma 8]{MonmarcheGamma}, for all $M\in\Sp$ and for all $h>0$, denoting by $h_t = e^{t L_{C,D}} h$, 
\begin{eqnarray*}
\partial_t \po \alpha''(h_t) \po \na h_t\pf^T M \na h_t\pf & \leq & 2\alpha''(h_t) \po \na h_t\pf^T M C^T \na h_t
\end{eqnarray*}
Applying this with $M=S^{-\frac12}QS^{-\frac12}$, we obtain
\begin{eqnarray*}
\partial_t \po \alpha''(h_t) \left|Q^{\frac12 }S^{-\frac12}\na h_t\right|^2\pf & \leq & -2\lambda \alpha''(h_t) \po \na h_t\pf^T S^{-\frac12}Q\po D + \widetilde J\pf S^{ -\frac12}\na h_t\\
&=& -\lambda \alpha''(h_t) \po S^{ -\frac12}\na h_t\pf^T \po QD+DQ +Q\widetilde J -\widetilde JQ\pf S^{-\frac12}\na h_t\\
&=& -2\lambda \alpha''(h_t) \po S^{-\frac12}\na h_t\pf^T Q S^{-\frac12}\na h_t
\end{eqnarray*}
where the first equality comes from the fact that $Q$ and $D$ are symmetric and $\widetilde J$ is antisymmetric, and the last equality from \eqref{inter}.
Hence
\begin{eqnarray*}
  \alpha''(h_t) \left|Q^{\frac12 }S^{-\frac12} h_t\right|^2  & \leq &  e^{-2\lambda (t-s)}  \alpha''(h_s) \left|Q^{\frac12 }S^{{-\frac12}} h_s\right|^2 .
\end{eqnarray*}
The log-Sobolev inequality for the standard Gaussian distribution $\gamma(\dd x) = \frac{1}{\sqrt{2\pi}}e^{-\frac12|x|^2}\dd x$ reads
\begin{eqnarray*}
 \ent_\gamma f & \leq & \frac12 \int \frac{|\na f|^2}{f} \dd \gamma,
\end{eqnarray*}
for all $f>0$ such that the r.h.s is finite. By the change of variable $z=S^{\frac12}x$ it yields
\begin{eqnarray*}
\ent_{\psi_\infty} f &\leq & \frac12 \int \frac{(\na f)^T S^{-1} \na f}{f} \dd \psi_\infty.
\end{eqnarray*}
Now
{
\begin{eqnarray*}
\ent_{\psi_\infty}\po h_t\pf & \leq & \frac{1 }{2} \int \frac{(\na h_t)^T S^{-1} \na h_t}{h_t} \dd \psi_\infty\\
& \leq & \frac{1}{2\nu_1}  \int \frac{ \left|Q^{\frac12 }S^{-\frac12}\na h_t\right|^2}{h_t} \dd \psi_\infty\\
& \leq & \frac{ e^{-2\lambda(t-s)}\nu_N }{2\nu_1} \int \frac{ \left|S^{-\frac12}\na h_s\right|^2}{h_s} \dd \psi_\infty,\\
& \leq & \frac{ \nu_N\ }{2\nu_1\min\sigma(S)} \,e^{-2\lambda(t-s)}\,\int \frac{ \left|\na h_s\right|^2}{h_s} \dd \psi_\infty.
\end{eqnarray*}
}
Finally the elliptic reversible generator $L_{-S,I_N}$ satisfies the Bakry-Emery criterion $\Gamma_2\geq 0$ (see~\cite{BGL} for instance, more precisely \cite[Equation 1.16.5 p. 72]{BGL} together with \cite[Theorem 5.5.2.(v), p.259]{BGL} with $\rho=0$ and integrated with respect to $\psi_\infty$), so that if now $h_t = e^{tL_{-S,I_N}} h$,
\begin{eqnarray*}
\int \frac{ \left|\na h_s\right|^2}{h_s} \dd \psi_\infty & \leq & \frac1{s} \ent_{\psi_\infty} h.
\end{eqnarray*}
We can then take $\nu_N$ arbitrarily close to $\nu_1$ to get the first part of Theorem \ref{TheoPrincipal2}. On the other hand, following \cite[Remark 8]{Lelievre2012}, if we choose $\nu_k = N+k$ (so that $\nu_N \leq 2\nu_1$), using that for any $T\in\Sp$, $\| AT\|_F \leq \max\sigma(T)\| A\|_F$, we get
\begin{eqnarray*}
\| A\|_F^2 & = & \lambda^2 \| S^{-\frac12}  \po  D+\widetilde J \pf S^{\frac12}\|_F^2\\
& \leq & \frac{\lambda^3 }{\min\sigma(S)}  \po N +  \sum_{j\neq k}\po {\frac{\nu_k+\nu_j}{\nu_k-\nu_j}}+1\pf^2 \pf\\
&\leq & \frac{\lambda^3 }{\min\sigma(S)}\po N+ N(N-1)(4N)^2\pf\\
&\leq & \frac{16\lambda^3 N^4 }{\min\sigma(S)}.
\end{eqnarray*}
\end{proof}

Note that an optimal $(A,D)$ is explicitly constructed in this proof. The procedure may be decomposed in two steps: first, exhibit a normalized eigenvector $v$ of $S$ associated to $\max\sigma(S)$, and set $D=N v v^T$. The second step answers the following general question: given any $D\in\Spe$, and writing $\widetilde D = (S)^{\frac12} D (S)^{\frac12}$, how to construct an optimal $\widetilde J\in\An$ so that
\begin{eqnarray*}
 \rho\po - \widetilde D - \widetilde  J\pf  &=& \frac{\text{Tr}\widetilde D}{N}\ ?
\end{eqnarray*}
By a straightforward adaptation of \cite[Algorithm p. 252]{Lelievre2012} (just change $S$ by $\widetilde D$ everywhere), construct an orthonormal basis $(\psi_1,\dots,\psi_N)$ such that $\psi_i^T\widetilde  D \psi_i = \text{Tr}\widetilde D/N$ for all $i$. Let $P$ be the matrix whose columns are the $\psi_i$'s, $\nu_k = N+k$ for $k\in\cco 1,N\ccf$ and $\widehat J$ be the antisymmetric matrix with coefficients ${\frac{\nu_k+\nu_l}{\nu_k-\nu_l}} \psi_k^T\widetilde D\psi_l$ for $k\neq l$. Then
\[ \widetilde  J :=  P \widehat J P^{-1}\]
is a solution of the problem, and we conclude by setting $A= -DS - \po S\pf^{-\frac12}{\widetilde J}\po S\pf^{\frac12}$.
\section{Numerical illustrations}

In dimension 2, consider $S=diag(\varepsilon,1)$. For any $h\in\R$, the corresponding $\psi_\infty$ is the unique equilibrium of
\begin{eqnarray}\label{EquaNumeriqueElliptique}
\dd X_t &=& - \begin{pmatrix}
\varepsilon & -h\\ \varepsilon h & 1
\end{pmatrix} X_t\dd t + \sqrt 2 \dd B_t
\end{eqnarray}
where $B=(B^1,B^2)$ is a 2-dimensional Brownian motion. It is also an equilibrium for 
\begin{eqnarray}\label{EquaNumeriqueHypoelliptique}
\dd Z_t &=& - \begin{pmatrix}
0 & -h\\ \varepsilon h & 2
\end{pmatrix} Z_t \dd t + \sqrt 2 \begin{pmatrix}
0\\ \dd  B^1+ \dd B^2
\end{pmatrix}
\end{eqnarray}
which is hypoelliptic as soon as $h\neq 0$. When $h$ is large enough, away from the origin (so that the random forces are small with respect to the deterministic drift), the behaviours of $X$ and $Z$ are similar, mostly driven by  a fast rotation, while the first coordinate of the reversible process solving \eqref{EquaNumeriqueElliptique} with $h=0$ moves slower, and thus covers the space less efficiently (see Fig.~\ref{Figure1}, where the parameter $p$ is the step size of the Euler Scheme). Note that in the case of $Z$, even if this rotation is randomly perturbed, since $\dd Z^1 = hZ^2 \dd t$, the process always goes from left to right in the lower half-plane $\{(x,y),\ y<0\}$ and from right to left in the upper one. 

\begin{figure}
 \centering \caption{Trajectories for a multi-scale equilibrium.}\label{Figure1}
\includegraphics[scale=0.7]{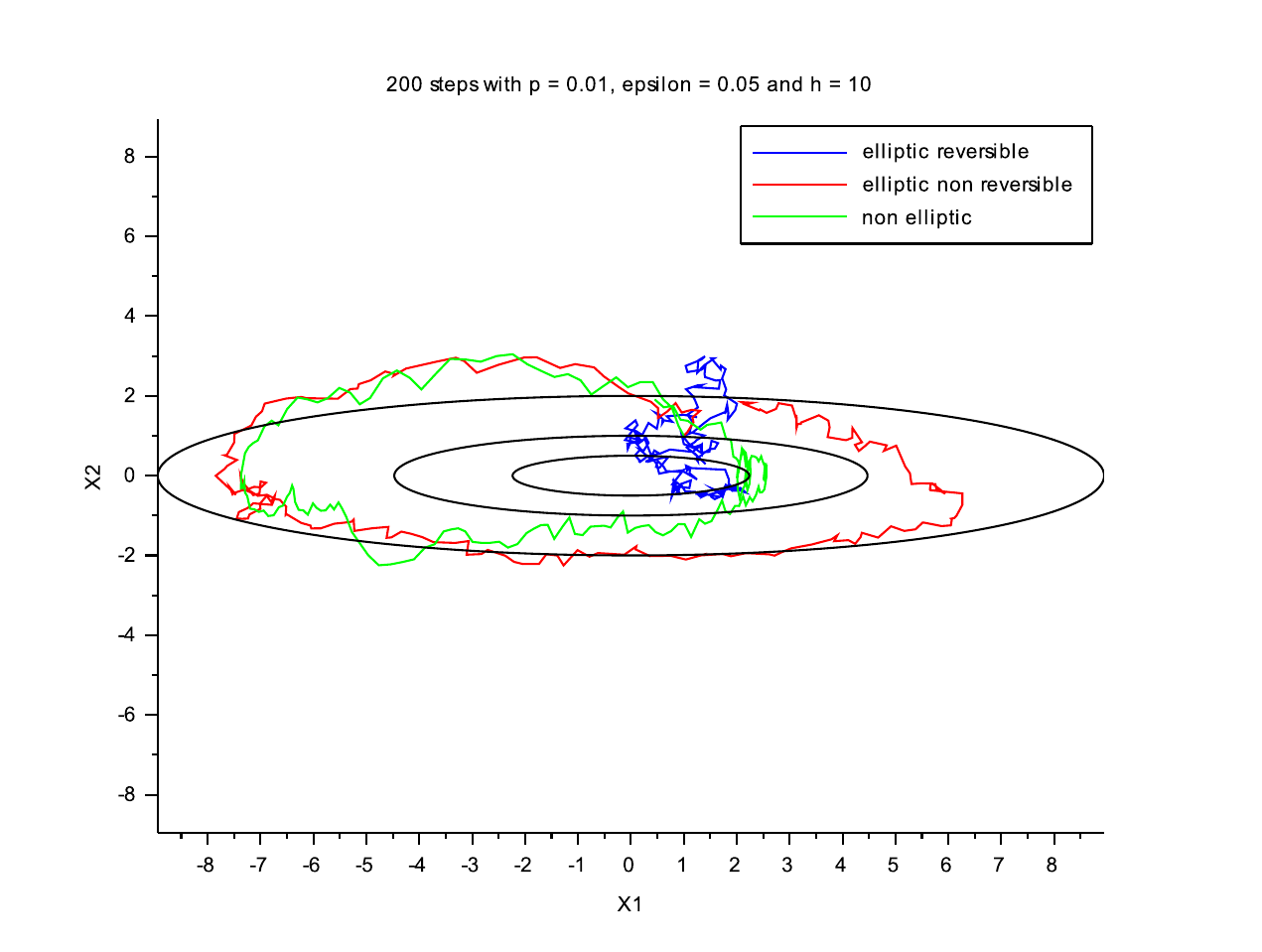}
\end{figure}

In \eqref{EquaNumeriqueElliptique}, the optimal rate $\frac{1+\varepsilon}{2}$ is obtained for $h^2 \geq \frac{(1+\varepsilon)^2}{4\varepsilon} - 1$ while in \eqref{EquaNumeriqueHypoelliptique} the optimal rate 1 is obtained for $h^2\geq \frac{1}{\varepsilon}$. For instance if we chose $h = \sqrt{\frac2\varepsilon}$, then both conditions are fulfilled and  in both cases the drift matrix is diagonalizable with two conjugated distinct eigenvalues. For a diagonalizable $2\times 2$ matrix $A$ with eigenvalues $\lambda_1 = \bar \lambda_2 \neq  \lambda_2$, denoting by $\nu = |\lambda_1-\lambda_2|$ and by $\alpha = |\bar v_1^T v_2|^{-1}$ where $(v_1,v_2)$ is a normalized eigenbasis of $A$, the Hermitian matrix norm of $e^{tA}$ can be explicitly computed (see e.g. \cite[Lemma 3]{Volte-Face}) as
\begin{eqnarray*}
\left\| e^{tA}\right\|^2 &=& e^{2 \Re(\lambda_1) t}\po 1 + \frac{2}{\sqrt{\frac{2(\alpha^2 -1)}{1- \cos(\nu t)}+1}-1}\pf.
\end{eqnarray*}
This is represented in Figure \ref{Figure4} (with $\varepsilon=0.05$ for every curves, $h=\sqrt{\frac2\varepsilon}$ for the second and third ones and $h=\sqrt{\frac1\varepsilon}$ for the last one). A large $h$ seems to improve the prefactor, and indeed, note that
\begin{eqnarray*}
M(h) \ :=\ \underset{t\geqslant 0} \max\po 1 + \frac{2}{\sqrt{\frac{2(\alpha^2 -1)}{1- \cos(\nu t)}+1}-1}\pf &=&  1+\frac{2}{\alpha -1}
\end{eqnarray*} 
and that, normalizing the eigenvectors $v_j = (-h,\lambda_j)$ for $j=1,2$, we can compute
\[ \alpha^{-2} \ =\  \frac{\left| h^2 + \po 1 + i \sqrt{h^2\varepsilon - 1}\pf^2\right|^2}{\po h^2 + 1 + h^2 \varepsilon - 1\pf^2} 
\ =\ \frac{ \po 1 - \varepsilon\pf^2 + 4   h^{-2}}{\po1+\varepsilon\pf^2}, \]
which decreases with $h^2$, together with the prefactor. As $h$ goes to infinity, $M(h)$ goes to its minimum, which is $\frac{1}{\varepsilon}$ (if $\varepsilon\leq 1$).
 
\begin{figure}
 \centering \caption{Norms of the drift matrix exponentials}\label{Figure4}
\includegraphics[scale=0.7]{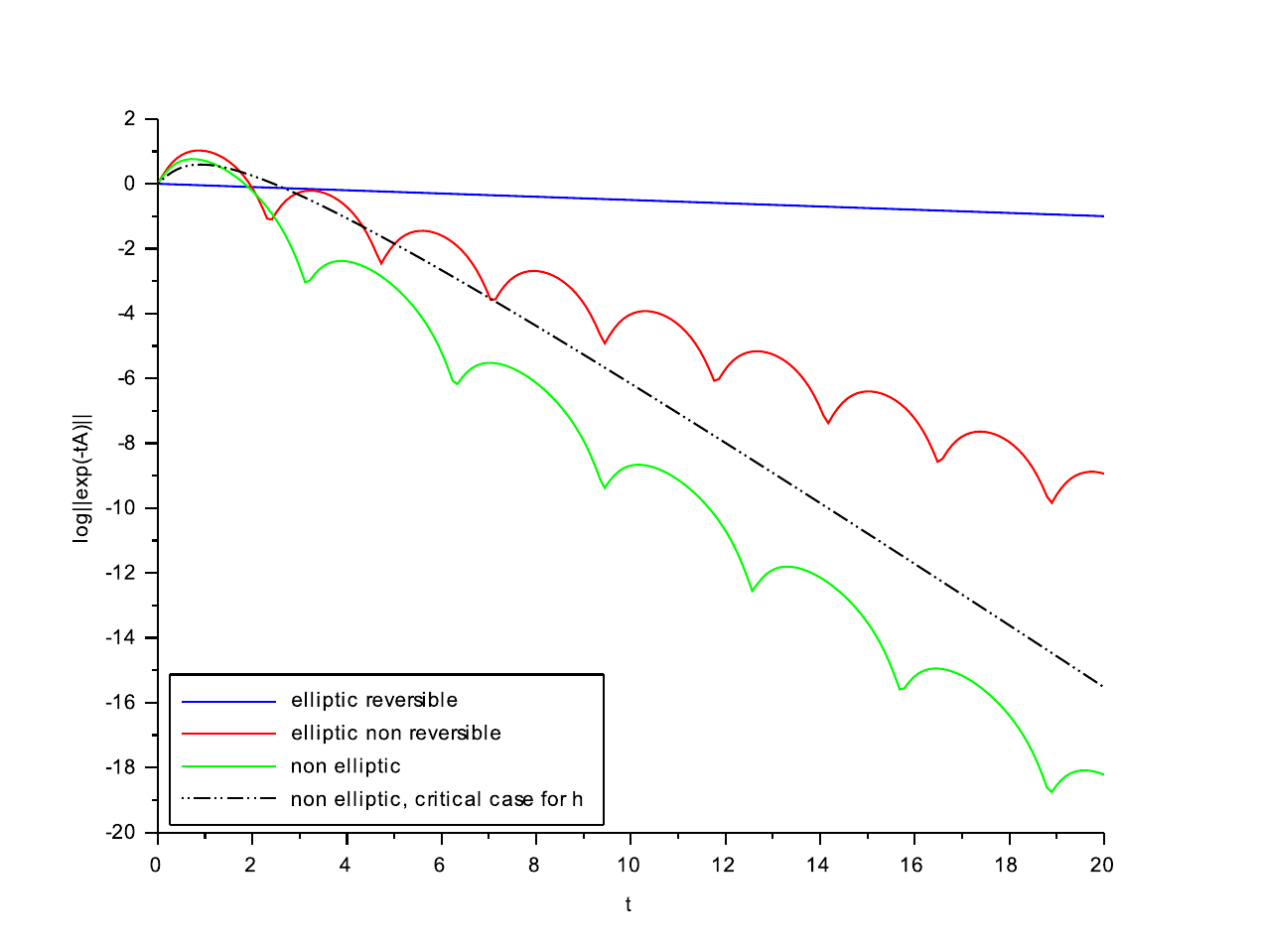}
\end{figure}
\bigskip

As remarked in the Introduction, how to boost the speed of convergence of a Markov process to sample a Gibbs measure is perhaps less relevant for quadratic potentials. There are thus two distinct problems that we have in mind for the future. The first one deals with the direct generalization of our result when we replace the Gaussian measure $\psi_\infty$ by $e^{-V}$ where $V$ satisfies $Hess(V)\ge S>0$, with $S$ constant positive symmetric. Is it possible to add a divergence free drift and a potentially degenerate (constant) diffusion matrix so that the rate of convergence to equilibrium is $\max(\sigma(S))$? The question is also of great interest when the dynamics is metastable, namely when $V$ has several local minima. A toy problem of this phenomenon would be to consider in dimension 1 (even if MCMC algorithm usually outperforms deterministic algorithms only in large dimension) a two-wells potential
\begin{eqnarray*}
V(x) &=& a x^4 - b x^2
\end{eqnarray*}
with $a,b>0$. Then $V$ has two minima $-b^2/(4a)$ attained at $x=\pm \sqrt{b/(2a)}$ and separated by a local maxima 0 at $x=0$. Depending on the energy barrier $V(0)-V\po\sqrt{b/(2a)}\pf = b^2/(4a)$  to overcome in order to go from one catchment area to the other, the reversible Fokker-Planck diffusion \eqref{EquaFP} will take a long time to achieve such a crossing. In Figures \ref{Figure2} and \ref{Figure3} are represented two such trajectories over different periods, along with trajectories of the first coordinate of a kinetic Langevin diffusion \eqref{EquaLangevin} (in each case both the reversible and the kinetic diffusions are driven by the same Brownian motion). As discussed at the end of Section \ref{SectionMainResult}, for the Langevin process there should be another parameter to tune, the variance $\nu$ of the velocity at equilibrium. Since the optimal choice of $\nu$ is already non trivial in the quadratic case, we won't address it here in this metastable context: for now, our considerations are only qualitative.

\begin{figure}
 \centering\caption{Metastable trajectories in short time.}\label{Figure2}
\includegraphics[scale=0.7]{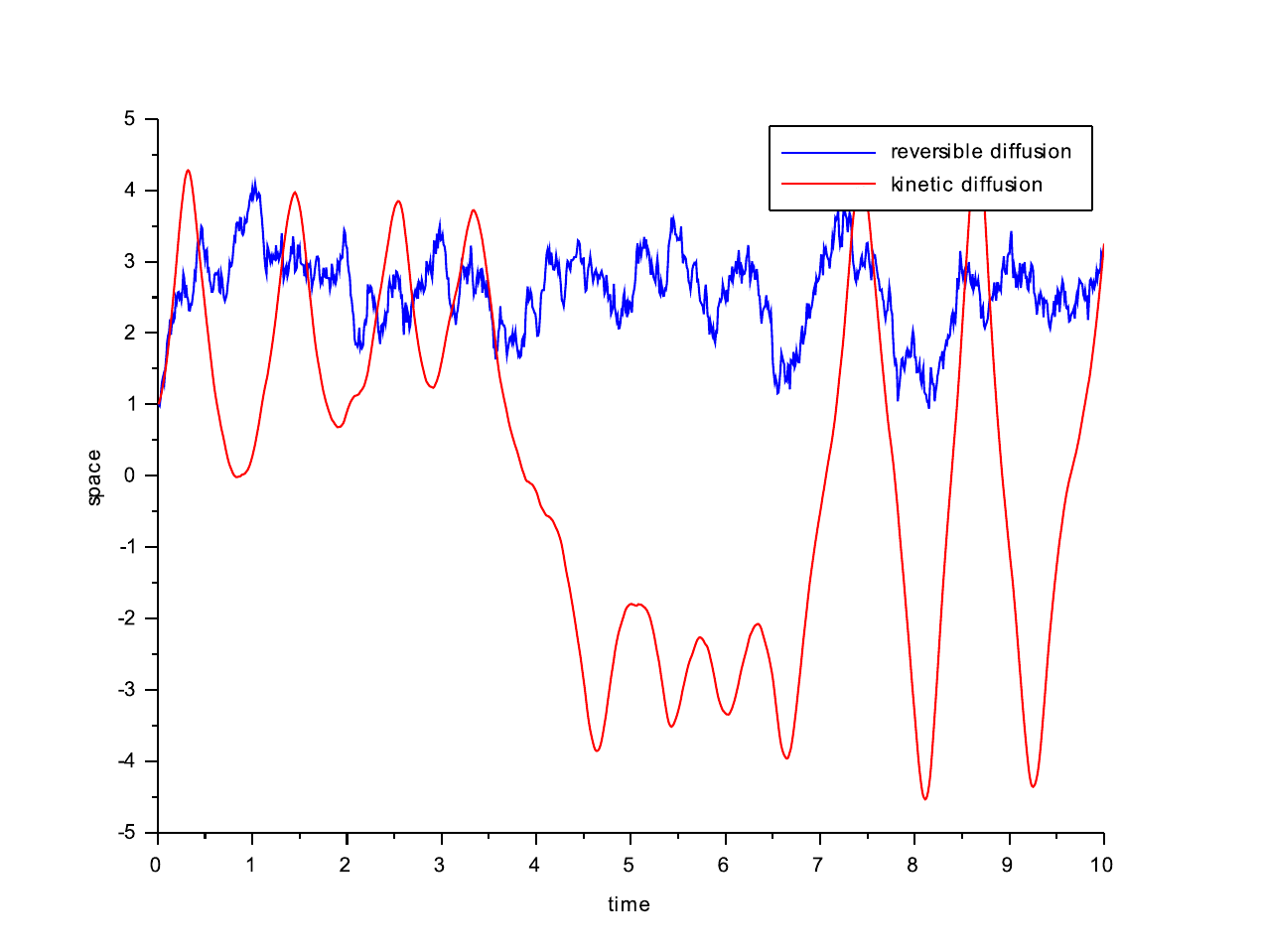}
\end{figure}

\begin{figure}
 \centering \caption{Metastable trajectories in longer time.}\label{Figure3}
\includegraphics[scale=0.7]{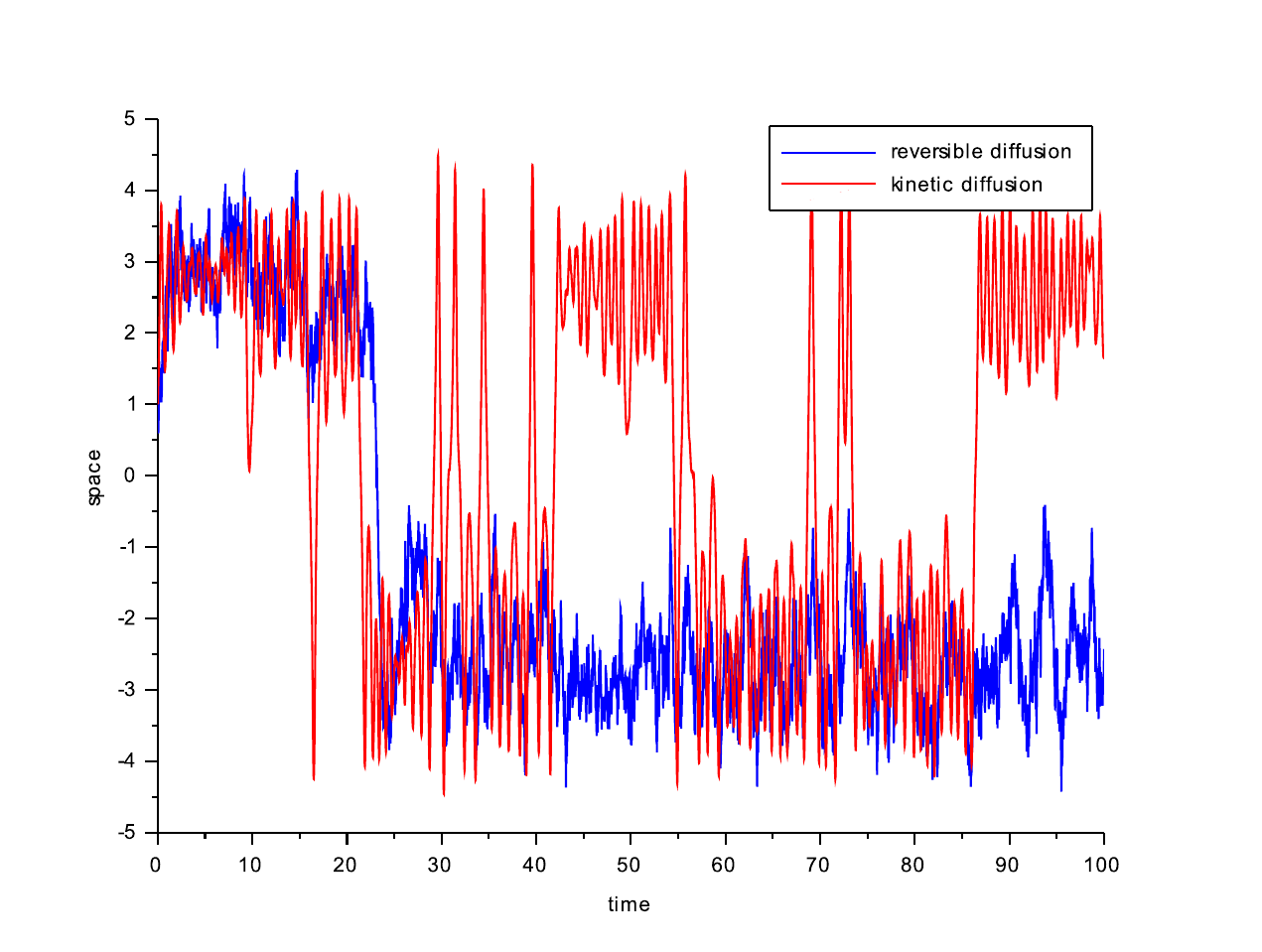}
\end{figure}

The first point to comment in Figure \ref{Figure2} is that the trajectory is smoother in the kinetic case than in the reversible one, which is obvious since in the first case it is 3/2-Hölder continuous while in the second one it is only 1/2-Hölder continuous. Second, due to its inertia, the trajectory in the kinetic case shows large oscillations in which kinetic and potential energy successively convert one to the other. In particular from the times $t\simeq 6.5$ to $10$ we can see the process has a high level of total energy and thus these large oscillations cross the energy barrier at $x=0$ without difficulty. At some point the total energy will decrease sufficiently for the process to stay trapped in the vicinity of one of the two minima, which has then a reasonable chance to be different from the one from which it started before the energy level got high.\\

That way we would interpret Figure \ref{Figure3} as an illustration to the fact the Langevin dynamics deals more efficiently with metastability (or at least energy barriers) than the reversible Fokker-Planck diffusion. With Theorem \ref{TheoPrincipal2} in mind, we could also interpolate from these figures the behaviour of a process that switch at random times from Equation \eqref{EquaNumeriqueElliptique} to \eqref{EquaNumeriqueHypoelliptique} (or anything else in that spirit). However it is difficult to export an intuition based on a toy model in dimension 1 and with a fixed set of parameter (especially the variance of the velocity in \eqref{EquaNumeriqueHypoelliptique}) to a more general case.

\vspace{20pt}

\noindent{\bf Acknowledgements}\\
The authors would like to thank the referees for well pointed remarks which have led to a significant improvement of the presentation of the paper.

\end{document}